\documentclass[reqno]{amsart}
\usepackage{amsmath}
\usepackage{amssymb, hyperref}

\usepackage{graphicx,color}
\usepackage{slashed}
\usepackage{mathtools}

\usepackage{parskip}
\mathtoolsset{showonlyrefs}
\usepackage{xcolor}

\begin{document}
\title[Inhomogeneous Hartree equation with a potential]{Scattering for the non-radial inhomogeneous Hartree equation with a potential}

 \author[C. M. GUZM\'AN ]
	{CARLOS M. GUZM\'AN } 
	
	\address{CARLOS M. GUZM\'AN \hfill\break
		Department of Mathematics, Fluminense Federal University, BRAZIL}
	\email{carlos.guz.j@gmail.com}
     
     \author[S. SILVA ]
	{SUERLAN SILVA}  
	
	\address{SUERLAN SILVA  \hfill\break
	Department of Mathematics, Fluminense Federal University, BRAZIL}
	\email{suerlansilva@id.uff.br}

      \author[G. PEÇANHA ]
	{GABRIEL PEÇANHA}  
	
	\address{GABRIEL PEÇANHA  \hfill\break
	Department of Mathematics, Fluminense Federal University, BRAZIL}
	\email{gabrielcavalcanti@id.uff.br}

\begin{abstract} In this work, we consider the focusing generalized inhomogeneous Hartree equation with potential
\[
i u_t + \Delta u - V(x)u + \left(I_{\gamma} * |x|^{-b}|u|^{p}\right)|x|^{-b}|u|^{p-2}u = 0,
\]
where \(0<\gamma<3\) and \(0<b<\tfrac{1+\gamma}{2}\). 
We prove scattering in the intercritical case for nonradial initial data, under a mass–potential condition that generalizes the usual mass–energy threshold. 
The main new points compared to previous works are the inhomogeneous weight \(|x|^{-b}\) and the presence of a potential \(V\), which lead us to study the perturbed operator \(-\Delta + V\). 

Our proof follows the general strategy of Murphy~\cite{MurphyNonradial}, but we need to adapt several steps to deal with the weight and the potential. We use Tao’s scattering criterion together with localized Morawetz estimates in this setting. As a preliminary step, we establish global well-posedness for small data, which, in the presence of \(V\), requires careful analysis using appropriate admissible Strichartz pairs.

\ 

\noindent Mathematics Subject Classification. 35A01, 35QA55, 35P25.
\end{abstract}

\keywords{Inhomogeneous generalized Hartree equation; Kato class potential; Virial Morawetz estimate; Scattering theory}

	\maketitle  
	\numberwithin{equation}{section}
	\newtheorem{theorem}{Theorem}[section]
	\newtheorem{proposition}[theorem]{Proposition}
	\newtheorem{lemma}[theorem]{Lemma}
	\newtheorem{corollary}[theorem]{Corollary}
	\newtheorem{remark}[theorem]{Remark}
	\newtheorem{definition}[theorem]{Definition}

\section{Introduction}

\ In this work, we consider the Cauchy problem for the focusing inhomogeneous generalized Hartree equation with potential:
\begin{equation}\label{IGHP}
\begin{cases}
iu_t + \Delta u - V(x)u + \left(I_{\gamma} * |x|^{-b} |u|^{p} \right) |x|^{-b} |u|^{p-2} u = 0, \\
u(0) = u_0 \in H^1(\mathbb{R}^3),
\end{cases}
\end{equation}
where \( u : \mathbb{R} \times \mathbb{R}^3 \to \mathbb{C} \), \( u_0 : \mathbb{R}^3 \to \mathbb{C} \), and \( * \) denotes convolution in \( \mathbb{R}^3 \). The nonlocal nonlinearity involves the inhomogeneous weight \( |x|^{-b} \) with \( b > 0 \), and the Riesz potential \( I_\gamma : \mathbb{R}^3 \to \mathbb{R} \) is defined by
\begin{equation}\label{riezpotential}
I_\gamma(x) := \frac{\Gamma\left(\frac{3 - \gamma}{2}\right)}{\Gamma\left(\frac{\gamma}{2}\right) \pi^{3/2} 2^\gamma |x|^{3 - \gamma}} =: \frac{\mathcal{K}}{|x|^{3 - \gamma}}, \quad 0 < \gamma < 3.
\end{equation}

\indent We assume that the external potential \( V : \mathbb{R}^3 \to \mathbb{R} \) satisfies the structural conditions
\begin{equation}\label{V-potential1}
V \in \mathcal{K}_0 \cap L^{3/2}(\mathbb{R}^3),
\end{equation}
and
\begin{equation}\label{V-potential2}
\left\| V_{-} \right\|_{\mathcal{K}} < 4\pi,
\end{equation}
where \( V_{-} := \min\{V, 0\} \) denotes the negative part of \( V \), and the Kato class \( \mathcal{K}_0 \) is the closure of bounded, compactly supported functions with respect to the Kato norm:
\begin{equation}\label{katonorm}
\| V \|_{\mathcal{K}} := \sup_{x \in \mathbb{R}^3} \int_{\mathbb{R}^3} \frac{|V(y)|}{|x - y|} \, dy.
\end{equation}

\indent Under assumptions \eqref{V-potential1}–\eqref{V-potential2}, the operator \( \mathcal{H} := -\Delta + V \) is self-adjoint with no eigenvalues. Moreover, the associated Schrödinger propagator \( e^{-it\mathcal{H}} \) satisfies dispersive and Strichartz estimates, as shown in~\cite{HONG}. The energy functional associated with \( \mathcal{H} \) is defined by
\[
\| \Lambda u \|^2 := \int_{\mathbb{R}^3} |\nabla u(x)|^2 \, dx + \int_{\mathbb{R}^3} V(x) |u(x)|^2 \, dx,
\]
which defines a norm equivalent to the standard \( H^1(\mathbb{R}^3) \)-norm under the conditions above.

\ The equation~\eqref{IGHP} enjoys the following scaling symmetry:
\begin{equation}
u_\lambda(t, x) = \lambda^{\frac{2 - 2b + \gamma}{2(p - 1)}} u\left(\lambda^2 t, \lambda x\right), \quad \lambda > 0,
\end{equation}
which identifies the unique invariant homogeneous \( L^2 \)-based Sobolev space as \( \dot{H}^{s_c}(\mathbb{R}^3) \), where
\[
s_c = \frac{3}{2} - \frac{2 - 2b + \gamma}{2(p - 1)}.
\]
When \( s_c = 0 \), the critical space is \( L^2(\mathbb{R}^3) \), which is naturally associated with the conservation of mass, given by
\begin{equation}\label{mass}
M(u) := \int_{\mathbb{R}^3} |u(t,x)|^2 \, dx = M(u_0).
\end{equation}

On the other hand, when \( s_c = 1 \), the critical space becomes \( \dot{H}^1(\mathbb{R}^3) \), corresponding to the conservation of energy, defined as the sum of kinetic and potential components:
\begin{equation}\label{energy}
E(u) := \frac{1}{2} \int_{\mathbb{R}^3} |\nabla u|^2 \, dx + \frac{1}{2} \int_{\mathbb{R}^3} V(x) |u|^2 \, dx - \frac{1}{p} P(u(t)) = E(u_0),
\end{equation}
where the nonlinear potential energy is given by
\begin{equation}
P(u(t)) := \int_{\mathbb{R}^3} \left( I_\gamma * |x|^{-b} |u|^p \right) |x|^{-b} |u|^p \, dx.
\end{equation}

Throughout this paper, it will also be useful to consider the energy functional in the absence of the external potential \( V \), defined by
\begin{equation}\label{energywithoutpotential}
E_0(u) := \frac{1}{2} \int_{\mathbb{R}^3} |\nabla u|^2 \, dx - \frac{1}{p} P(u(t)).
\end{equation}

\ In this paper, we establish scattering for nonradial data in the intercritical range up to 
\( p < \frac{5}{2} - 2b + \gamma \), slightly below the \( \dot{H}^1 \)-critical threshold 
\( p < 3 - 2b + \gamma \), due to technical constraints in the nonlinear estimates required for the small-data theory, which plays a crucial role in our proof.

\begin{definition}
We say that a solution \( u \) to~\eqref{IGHP} \emph{scatters forward in time in} \( H^1(\mathbb{R}^3) \) if there exists \( u_+ \in H^1(\mathbb{R}^3) \) such that
\begin{equation}
\lim_{t \to \infty} \left\| u(t) - e^{-it\mathcal{H}} u_+ \right\|_{H^1(\mathbb{R}^3)} = 0,
\end{equation}
where \( e^{-it\mathcal{H}} \) is the linear Schrödinger propagator associated with \( \mathcal{H} = -\Delta + V(x) \).
\end{definition}

\ Before stating our result, we recall some key developments in the literature. The well-posedness of the inhomogeneous Hartree equation without potential (i.e., \( V = 0 \), \( p = 2 \), and \( b \neq 0 \) in equation~\eqref{IGHP}) was first established by Cazenave~\cite{CAZENAVEBOOK}. For the generalized Hartree equation with \( p > 2 \), global well-posedness and scattering in \( H^1(\mathbb{R}^3) \) were proved by Arora and Roudenko~\cite{10.1307/mmj/20205855}, using the concentration–compactness and rigidity method of Kenig–Merle~\cite{KENIG}. Later, Arora provided an alternative proof~\cite{Anudeep} via the Dodson–Murphy approach~\cite{Dod-Mur}, which relies on Tao’s scattering criterion~\cite{Tao} and radial Morawetz-type estimates. Feng and Yuan~\cite{Binhua} studied the subcritical Hartree equation with \( V = 0 \) and \( b = 0 \). Saanouni and Chengbin~\cite{Tarek} established scattering for radial data in the case \( V = 0 \), \( b > 0 \), and Chengbin~\cite{xu2021scatteringnonradialfocusinginhomogeneous} extended the result to the nonradial setting. More recently, Guzmán-Loli-Yapu~\cite{guzman_loli_yapu_2024} established scattering for the homogeneous Hartree equation ($b=0$) with \( V \neq 0 \), assuming initial radial data.

\ Motivated by the works mentioned above, we extend the analysis to the inhomogeneous generalized Hartree equation with potential in the nonradial setting. Our approach draws on techniques inspired by Murphy~\cite{MurphyNonradial}. 
As a first step, we establish global existence for small initial data in \( H^1(\mathbb{R}^3) \), which serves as a starting point for developing the scattering theory.


\begin{proposition}\label{GWPH2} Let $0<b<\frac{1+\gamma}{2}$ and $\frac{5-2b+\gamma}{3}<p<\frac{5-4b+2\gamma}{2}$. Suppose that $\|u_0\|_{H^{1}(\mathbb{R}^{3})}\leq E$,  then there exists $\delta_{sd}=\delta_{sd}(E)>0$ such that, if 
\begin{equation}\label{smalldatawellposedness}
\|e^{-it\mathcal{H}}u_0\|_{S(\dot{H}^{s_c},[0,+\infty))}<\delta_{sd}  
\end{equation}
then, the solution $u$ to \eqref{IGHP} with initial condition $u_0\in H^{1}(\mathbb{R}^{3})$ is globally defined on $[0,+\infty)$. Moreover,
\begin{equation}\label{NGWP3}
\|u\|_{S(\dot{H}^{s_c},[0,+\infty))}\leq  2\delta_{sd}\quad \textnormal{and}\quad\|u\|_{S\left(L^2,[0,+\infty)\right)}+\|\Lambda u\|_{S\left(L^2,[0,+\infty)\right)}\leq 2cE.
\end{equation}
\end{proposition}


\begin{remark}\label{important remark1} It is worth mentioning that, in our global well-posedness result, we impose the restriction \( p < \frac{5}{2} - 2b + \gamma \), slightly below the full intercritical range \( p < 3 - 2b + \gamma \). This condition stems from the use of the \( \| \cdot \|_{S(\dot{H}^{s_c})} \) norm, which controls the scattering size. In our argument, it naturally arises when selecting admissible Strichartz pairs that satisfy the required parameter constraints. We emphasize this aspect, as it is crucial for establishing the scattering result, particularly in the presence of a potential, where the analysis demands greater precision. 
\end{remark} 

\  We are now ready to state the main result of this paper, which establishes scattering for general initial data to the Cauchy problem~\eqref{IGHP}.

\begin{theorem}\label{mainresulttheorem}
Let  \( 0 < b < \frac{1 + \gamma}{2} \), and $\frac{5 - 2b + \gamma}{3} < p < \frac{5 - 4b + 2\gamma}{2}$. Let \( V : \mathbb{R}^3 \rightarrow \mathbb{R} \) satisfy \eqref{V-potential1}, with 
\begin{equation}
 V \geq 0,\quad x \cdot \nabla V \leq 0,\quad \textnormal{and}\quad  x \cdot \nabla V \in L^{r},
\end{equation}
for some \( r \in [\frac{3}{2}, \infty) \). Let \( u_0 \in H^1(\mathbb{R}^3) \), possibly non-radial, and let \( u(t) \) be the corresponding solution to \eqref{IGHP}, satisfying
\begin{equation}\label{newcondition}
\sup_{t \in [0, T^*]} P(u(t)) M(u(t))^{\sigma_c} < P(Q) M(Q)^{\sigma_c}, 
\quad \text{where} \quad \sigma_c = \frac{1 - s_c}{s_c}.
\end{equation}
Then \( u \) is a global solution to \eqref{IGHP} and scatters in \( H^1(\mathbb{R}^3) \).
\end{theorem}

\begin{remark}
The hypotheses in Theorem~\ref{mainresulttheorem} involve quantities that are \emph{not conserved} along the flow of equation~\eqref{IGHP}, particularly the nonlinear potential energy \( P(u(t)) \). This introduces additional complexity in the analysis, especially when \( V(x) \neq 0 \), where classical conservation laws must be supplemented with robust \emph{a priori} estimates. In global theory, we make use of both the \( H^{s_c} \)-admissible pair \( (\bar{a}, \bar{r}) \) and the \( L^2 \)-admissible pair \( (4^+, 3^-) \). Note that \( 3^- < 3 \), which is necessary for the equivalence of Sobolev norms. To satisfy the condition \( \bar{r} < 6 \), required for the admissibility of \( (\bar{a}, \bar{r}) \), we are led to impose the restriction $p < \frac{5}{2} - 2b + \gamma.$ (See Section $2$).
\end{remark}

\begin{remark}
After completing this manuscript, we became aware of an independent work by Saanouni and Salah~\cite{B_Saanouni_2025}, which also addresses the scattering problem. Although they claim to treat scattering for the full intercritical range, several foundational aspects, such as the global small data theory, the non-linear estimates required for global control, and the role of the scattering norm $\| \cdot \|_{S(\dot{H}^{s_c})}$, are not fully detailed. In particular, based on the previous remark, it is not yet clear to us how the result in~\cite{B_Saanouni_2025} extends up to the upper limit of the intercritical range, $3-2b+\gamma$. Our work aims to provide a more complete treatment of these core components, thereby ensuring a rigorous justification of the scattering result.
\end{remark}


\ Our proof of Theorem~\ref{mainresulttheorem} adopts a direct and quantitative strategy based on a localized virial/Morawetz estimate, thereby avoiding the contradiction framework of Kenig–Merle. We apply a version of Tao's scattering criterion: if a global solution \( u \in L^\infty_t H^1_x \) satisfies
\[
\lim_{t \to \infty} \int_{|x| \leq R} |u(t,x)|^2 \, dx = 0,
\]
for some (and hence any) \( R > 0 \), then \( u \) scatters in \( H^1(\mathbb{R}^3) \). Given the potential–mass threshold assumption in the theorem, we establish that this local decay condition is satisfied. The key step is to derive a localized Morawetz estimate tailored to the presence of the external potential \( V(x) \), taking advantage of its positivity and repulsive character, along with the structure of the non-local non-linearity. These ingredients yield uniform-in-time control of mass and energy inside compact regions, which ensures the decay required by the scattering criterion and concludes the proof.

\noindent Before outlining the main ideas behind the proof of Theorem~\ref{mainresulttheorem}, we record a useful consequence of the threshold assumption. The result below provides an a priori control along the flow in terms of the conserved mass and energy, and reflects the standard conditions considered in earlier works.

\begin{corollary}\label{maincorollary}
Let \( u \) be a solution to \eqref{IGHP} with maximal lifespan. If
\begin{equation}\label{1}
    M(u_0)^{\sigma_c}E(u_0) < M(Q)^{\sigma_c}E_0(Q),
\end{equation}
and
\begin{equation}\label{2}
    \|u_0\|_{L^{2}}^{\sigma_c} \|\Lambda u_0\|_{L^{2}} < \|Q\|_{L^{2}}^{\sigma_c} \|\nabla Q\|_{L^{2}},
\end{equation}
then
\begin{equation}
\sup_{t \in [0, \infty)} \|u(t)\|^{\sigma_{c}}_{L^{2}} \|\Lambda u(t)\|_{L^{2}} < \|Q\|^{\sigma_{c}}_{L^{2}} \|\nabla Q\|_{L^{2}}.
\end{equation}
Consequently, \( u \) scatters in \( H^1(\mathbb{R}^3) \).
\end{corollary}


\subsection{Outline}

\ The paper is organized as follows. In Section~2, we present preliminary results and establish the well-posedness theory. In Section~3, we prove the scattering criterion following the method of Murphy~\cite{MurphyNonradial} (see also~\cite{CardosoCampos}). Section~4 is devoted to the variational analysis of the ground state and the proof of Corollary~\eqref{maincorollary}. Finally, in Section~5, we prove Theorem~\eqref{mainresulttheorem} using localized virial–Morawetz estimates.

\section{\bf Notation and Preliminaries}\label{sec2}

\ We begin by introducing the notation and functional framework used throughout the paper. We write \( a \lesssim b \) to indicate \( a \leq c b \) for some constant \( c > 0 \), with dependencies specified via subscripts when relevant.

Let \( I \subset \mathbb{R} \) and \( p, r \in [1, \infty] \). We use the standard Lebesgue spaces \( L^p(\mathbb{R}^3) \) and the mixed space-time spaces \( L^p_t L^r_x(I \times \mathbb{R}^3) \), defined by
\[
\|f\|_{L^p(I; L^r)} := \left( \int_I \left( \int_{\mathbb{R}^3} |f(t,x)|^r dx \right)^{p/r} dt \right)^{1/p} < \infty.
\]
When \( I = \mathbb{R} \), we omit the interval and simply write \( \| \cdot \|_{L^p_t L^r_x} \).

We also work with the homogeneous and inhomogeneous Sobolev spaces \( \dot{H}^{s,r}(\mathbb{R}^3) \) and \( H^{s,r}(\mathbb{R}^3) \). In the case \( r = 2 \), we abbreviate \( \dot{H}^s := \dot{H}^{s,2} \) and \( H^s := H^{s,2} \). Subscripts are used when the variable of integration needs to be emphasized.

\subsection{Strichartz Estimates}

To develop the well-posedness theory for \eqref{IGHP}, we recall the Strichartz estimates in the form suited to our setting. A pair \( (q, r) \) is:

\begin{itemize}
    \item \( L^2 \)-admissible if \( \frac{2}{q} = \frac{3}{2} - \frac{3}{r} \), with \( 2 \leq r \leq 6 \);
    \item \( \dot{H}^s \)-admissible if \( \frac{2}{q} = \frac{3}{2} - \frac{3}{r} - s \);
    \item \( \dot{H}^{-s} \)-admissible if \( \frac{2}{q} = \frac{3}{2} - \frac{3}{r} + s \).
\end{itemize}

Given \( s \in (0,1) \), we define the admissible sets:
\[
A_s := \left\{ (q, r)\ \text{is } \dot{H}^s\text{-admissible};\ \left(\tfrac{2N}{N - 2s} \right)^+ \leq r \leq 6^- \right\}.
\]
Note that \( A_0 \) denotes the set of \( L^2 \)-admissible pairs. Here, \( a^\pm = a \pm \varepsilon \), for a small \( \varepsilon > 0 \).

We define the Strichartz norm:
\[
\|f\|_{S(\dot{H}^s, I)} := \sup_{(q, r) \in A_s} \|f\|_{L^q_t L^r_x(I \times \mathbb{R}^3)},
\]
and the dual norm:
\[
\|f\|_{S'(\dot{H}^{-s}, I)} := \inf_{(q, r) \in A_{-s}} \|f\|_{L^{\hat{q}}_t L^{\hat{r}}_x(I \times \mathbb{R}^3)},
\]
where \( (\hat{q}, \hat{r}) \) are the Hölder conjugates of \( (q, r) \).

\ In the sequel, we discuss some inequalities that will be important in the paper.
\begin{lemma}[Dispersive estimate, \cite{HONG}] Let $V:\mathbb{R}^{3}\rightarrow\mathbb{R}$ satisfy (\ref{V-potential1}) and (\ref{V-potential2}). Then it holds that
\begin{equation}\label{DE1}
\|e^{-it\mathcal{H}}\|_{L^{1}\rightarrow L^{\infty}}\lesssim |t|^{-\frac{3}{2}}
\end{equation}

\end{lemma}
Thanks to this dispersive estimate and the abstract theory of Keel and Tao in \cite{:keel} (see also Foschi \cite{foschi2005inhomogeneous}), we have the following Strichartz estimates:
\begin{lemma}[Strichartz estimate, \cite{HONG}] Let $V:\mathbb{R}^{3}\rightarrow\mathbb{R}$ satisfy (\ref{V-potential1}) and (\ref{V-potential2}). Then it holds that  
\begin{equation}\label{CSE1}
\|e^{-it\mathcal{H}}f\|_{S(\dot H^{s})}\lesssim\|f\|_{\dot H^{s}};
\end{equation}
and

\begin{equation}\label{CSE3}
\left\|\int_{0}^{t}e^{-i(t-s)\mathcal{H}}f(\cdot,s)ds\right\|_{S(L^{2},I)}\lesssim\|f\|_{S^{'}(L^{2},I)}.
\end{equation}
\begin{lemma}[Kato inhomogeneous Strichartz estimate, \cite{HONG}] Let $V:\mathbb{R}^{3}\rightarrow\mathbb{R}$ satisfy (\ref{V-potential1}) and (\ref{V-potential2}). Then,
\begin{equation}\label{CSE2}
\left\|\int_{0}^{t}e^{-i(t-s)\mathcal{H}}f(\cdot,s)ds\right\|_{S(\dot H^{s},I)}\lesssim\|f\|_{S^{'}(\dot H^{-s},I)};
\end{equation}
\end{lemma}
\subsection{Equivalence of Sobolev norms} 
In this subsection, we define the homogeneous and inhomogeneous Sobolev spaces associated to $\mathcal{H}$ as the closure of $C^{\infty}_{0}(\mathbb{R}^{3})$ under the norms. Let us define $\Lambda=\mathcal{H}^{\frac{1}{2}}$ and $\langle\Lambda\rangle=(1+\mathcal{H})^{\frac{1}{2}}$ and the norms by
\begin{eqnarray}
\|f\|_{\dot W_{V}^{\alpha,r}}:=\|\Lambda^{\alpha}f\|_{L^{r}},\ \|f\|_{W_{V}^{\alpha,r}}:=\|\langle\Lambda \rangle^{\alpha}f\|_{L^{r}},
\end{eqnarray}
\begin{lemma}(Sobolev inequalities \cite{HONG})Let $V:\mathbb{R}^{3}\rightarrow\mathbb{R}$ satisfy (\ref{V-potential1}) and (\ref{V-potential2}). Then it holds that
\begin{equation}
\|f\|_{L^{q}}\lesssim \|f\|_{\dot W_{V}^{\alpha,r}}, \|f\|_{W_{V}^{\alpha,r}}\sim\|f\|_{W^{\alpha,r}}
\end{equation}
where $1<p,q<\infty, 1<p<\frac{3}{\alpha}$, $0\leq\alpha\leq2$ and $\frac{1}{q}=\frac{1}{p}-\frac{\alpha}{3}$
\end{lemma}
\begin{lemma}(Equivalence of Sobolev spaces \cite{HONG}). Let $V:\mathbb{R}^{3}\rightarrow\mathbb{R}$ satisfy (\ref{V-potential1}) and (\ref{V-potential2}). Then it holds that
\begin{equation}
\|f\|_{\dot W_{V}^{\alpha,r}}\sim \|f\|_{W^{\alpha,r}}\ \text{and} \ \|f\|_{W_{V}^{\alpha,r}}\sim\|f\|_{W^{\alpha,r}}
\end{equation}
where $1<r<\frac{3}{\alpha}$ and $0\leq\alpha\leq2$.
\end{lemma}

\end{lemma}
\begin{lemma}Let $N\geq$, $0<\lambda<3$, $1<r,s<\infty$ and $f\in L^{r}, g\in L^{s}$. If $\displaystyle\frac{1}{r}+\frac{1}{s}+\frac{\lambda}{3}=2$, then
\begin{equation}
\int\int_{\mathbb{R^N}\times\mathbb{R^N}}\frac{f(x)g(y)}{|x-y|^\lambda}dxdy\leq C(N,s,\lambda)\|f\|_{L^{r}_{x}}\|g\|_{L^{s}_{x}}
\end{equation}  
\end{lemma}
\begin{corollary} Let $0<\gamma<3$ and $1< s,q,r<\infty$ be such that $\displaystyle\frac{1}{q}+\frac{1}{r}+\frac{1}{s}=1+\frac{\gamma}{3}$. Assume that $f\in L^{s}(\mathbb{R}^{3})$ and $g\in L^{q}(\mathbb{R}^{3})$. Then,
$$\|(I_{\gamma}*f)g\|_{r'}\lesssim\|f\|_{s}\|g\|_{q}$$
\end{corollary}

\subsection{Local and Global Well-posedness}

In this section we study the well-posendess theory, we start with the local theory. Due to the equivalence of norms \( \|\Lambda u\|_{L^r} \sim \|\nabla u\|_{L^r} \), we require \( 1 < r < 3 \).

\begin{proposition}[Local well-posedness]\label{prop:local}
Suppose \( p \) satisfies \( 2 \leq p < 3 - 2b + \gamma \), and assume the potential \( V \) meets conditions \eqref{V-potential1}--\eqref{V-potential2}. Let \( u_0 \in H^1(\mathbb{R}^3) \). Then there exists \( T =T(\|u_0\|_{H^1}, p,\gamma) > 0 \), such that the Cauchy problem \eqref{IGHP} admits a unique solution \( u \) on \([0, T]\) with
\[
u \in C([0,T], H^1(\mathbb{R}^3)) \cap L^q([0,T], W^{1,r}_V(\mathbb{R}^3)),
\]
where \( (q, r) \) is \( L^2 \)-admissible.
\end{proposition}

\begin{proof}
The proof is similar to the case without potential, as the admissible pair with \( r < 3 \) ensures Sobolev norm equivalence. See \cite{alharbi_saanouni_2019}.
\end{proof}

\ To prove global existence in \( H^1 \), we proceed more carefully. In particular, to retain the equivalence \( \|\Lambda u\|_{L^r} \sim \|\nabla u\|_{L^r} \), it is necessary to choose admissible pairs \( (q, r) \) appropriately. Before presenting the global result, we establish the nonlinear estimates that serve as key analytic tools for proving both global well-posedness for small data and the scattering criterion.

\begin{lemma}[Nonlinear estimates]\label{L:NL} Let  $0<b<\frac{1+\gamma}{2}$ and $\frac{5-2b+\gamma}{3}<p<\frac{5-4b+2\gamma}{2}$ . There exist a small parameter $\theta>0$ such that
\begin{itemize}
\item [(i)] $\left \| N_{\gamma}(u)\right\|_{S'(\dot{H}^{-s_c})} \lesssim \| u\|^{2\theta}_{L^\infty_tH^{1}_x}\|u\|_{L^a L^r}^{2p-1-2\theta}$,
\item [(ii)] $\left\|N_{\gamma}(u)\right\|_{S'(L^2)}\lesssim \| u\|^{2\theta}_{L^\infty_tH^{1}_x}\|u\|^{2p-2-2\theta}_{ L^{\overline{a}} L^{\overline{r}}} \| u\|_{L^{4^{+}}L^{3^-}},
$
\item [(iii)] $\left\|\nabla N_{\gamma}(u)\right\|_{S'(L^2)} \lesssim 
\| u\|^{2\theta}_{L^\infty H^{1}_{x}}\|u\|^{2p-2-2\theta}_{L^{\overline{a}}L^{\overline{r}}} \|\nabla u\|_{L^{4^{+}}L^{3^-}}.
$

where $N_{\gamma}=(I_{\gamma}*|x|^{-b}|u|^{p})|x|^{-b}|u|^{p-2}u$,  \;  $3^-=\frac{3}{1+\varepsilon}$ and $ 4^+=\frac{4}{1-2\varepsilon}$.
\end{itemize}
\end{lemma}
\begin{proof} Here, we use the \( L^2 \)-admissible pair \( (4^{+}, 3^{-}) \), justified by the equivalence of Sobolev spaces and the condition \( r < 3 \).

\begin{itemize}
\item [(i)] By definition of the norm $S'(\dot{H}^{-s_c})$, we have $\| N_{\gamma}(u)\|_{S'(\dot{H}^{-s_c})} \lesssim \|N_{\gamma}(u)\|_{L^{\tilde{a}'}L^{r'}}$. 

Let $A \in \{B_{1}(0), B_{1}^{c}(0)\}$. Applying the Hardy–Littlewood–Sobolev inequality, we obtain
\begin{align*}
\| N_{\gamma}(u)\|_{L^{r'}(A)} 
& \lesssim  \||x|^{-b}|u|^p\|_{L^{\beta}(A)} \cdot \||x|^{-b}|u|^{p-2}u\|_{L^{\sigma}(A)} \\
& \lesssim \||x|^{-b}\|_{L^{\alpha}(A)} \|u\|_{L^r(A)}^{p-\theta} \|u\|_{L^{r_1}(A)}^{\theta} \cdot \||x|^{-b}\|_{L^{\alpha}(A)} \|u\|_{L^r(A)}^{p-2-\theta} \|u\|_{L^{r_1}(A)}^{\theta} \|u\|_{L^r(A)} \\
& \lesssim \|u\|_{L^r(A)}^{2p - 2\theta - 1} \|u\|_{L^{r_1}(A)}^{2\theta},
\end{align*}
where
\[
1 + \frac{\gamma}{3} = \frac{1}{r} + \frac{1}{\beta} + \frac{1}{\sigma}, \quad \frac{1}{\beta} = \frac{1}{\alpha} + \frac{p - \theta}{r} + \frac{\theta}{r_1}, \quad \frac{1}{\sigma} = \frac{1}{\alpha} + \frac{p - 1 - \theta}{r} + \frac{\theta}{r_1}.
\]

We must ensure the integrability of $|x|^{-b}$. Specifically, we need
\[
\frac{3}{\alpha} - b > 0 \quad \text{for } A = B_1(0) \quad \textnormal{and}\quad \frac{3}{\alpha} - b < 0 \quad \text{for } A = B_1^c(0).
\]
To this end, for the interior region \( A = B_1(0) \), we choose \( r_1 = 6 \) and set
\[
r = \frac{6(p - \theta)}{(3 - 2s_c)(p - \theta) - 2(1 - s_c)}.
\]
For the exterior region \( A = B_1^c(0) \), we choose \( r_1 = 2 \). Thus, in both cases, we can use the Sobolev embedding \( H^1 \subset L^{\theta r_1} \).

Furthermore, applying the norm \( \|\cdot\|_{L^{\tilde{a}'}_t} \), we get
\begin{align*}
\|N_{\gamma}(u)\|_{L^{\tilde{a}'}L^{r'}} 
& \lesssim \left\| \|u\|_{L^{r_1}}^{2\theta} \|u\|_{L^r}^{2p - 1 - 2\theta} \right\|_{L^{\tilde{a}'}} \\
& \lesssim \|u\|_{L^{\infty}H^{1}}^{2\theta} \|u\|_{L^a L^r}^{2p - 1 - 2\theta}.
\end{align*}

Note that, with the above value of \( r \), one has\footnote{Observe that \( a = \frac{2(p - \theta)}{1 - s_c} > \frac{2}{1 - s_c} \), implying \( r < 6 \), and hence \( (a, r) \) is \( \dot{H}^{s_c} \)-admissible.
}
\[
a = \frac{2(p - \theta)}{1 - s_c}, \qquad \tilde{a} = \frac{2(p - \theta)}{2(p - \theta)s_c + 1 - s_c},
\]
in view of the fact that \( (a, r) \) is \( \dot{H}^{s_c} \)-admissible and \( (\tilde{a}, r) \) is \( \dot{H}^{-s_c} \)-admissible. Moreover, we can verify that $(2p-1-2\theta)\tilde{a}'=a$.

        
 \item[(iii)] 
Observe that,
\begin{eqnarray*}
|\nabla N_{\gamma}(u)| &\lesssim& (I_{\gamma} * |x|^{-b-1}|u|^p) \, |x|^{-b}|u|^{p-1} + (I_{\gamma} * |x|^{-b}|u|^{p-1}|\nabla u|) \, |x|^{-b}|u|^{p-1} \\
&& + (I_{\gamma} * |x|^{-b}|u|^p) \, |x|^{-b-1}|u|^{p-1} + (I_{\gamma} * |x|^{-b}|u|^p) \, |x|^{-b}|u|^{p-2}|\nabla u|\\
&:=&N_1+N_2+N_3+N_4.
\end{eqnarray*}

We estimate the second term (\(N_2\)), as the others are analogous. For example, 
\[
N_1 = \left(|x|^{-b-1}|u|^p\right) \cdot \left(|x|^{-b}|u|^{p-1}\right) = \left(|x|^{-b}|u|^{p-1}\right) \cdot \left(|x|^{-1}|u|\right) \cdot \left(|x|^{-b}|u|^{p-1}\right).
\]
By Hardy's inequality, \( \||x|^{-1} u\| \lesssim \|\nabla u\| \), so \(N_1\) admits the same bound as \(N_2\). The terms \(N_3\) and \(N_4\) are treated in the same way.
\begin{eqnarray*}
\|N_2\|_{L^{\frac{6}{5}}}
&\lesssim& \| |x|^{-b}|u|^{p-1}|\nabla u|\|_{L^{\beta}} \cdot \||x|^{-b}|u|^{p-1}\|_{L^{\sigma}} \\
&\lesssim& \| |x|^{-b}\|_{L^j(A)} \cdot \||u|^{p-1}\|_{L^n} \cdot \|\nabla u\|_{3^-} \cdot \| |x|^{-b}\|_{L^j(A)} \cdot \||u|^{p-1}\|_{L^n} \\
&\lesssim& \|u\|_{L^{r_1\theta}}^{2\theta} \cdot \|u\|_{\overline{r}}^{2p-2-2\theta} \cdot \|\nabla u\|_{3^-},
\end{eqnarray*}
where the exponents satisfy
\[
\frac{5}{6}+\frac{\gamma}{3} = \frac{1}{\beta} + \frac{1}{\sigma}, \quad \frac{1}{\beta} = \frac{1}{j} + \frac{1}{n} + \frac{1}{3^-}, \quad \frac{1}{\sigma} = \frac{1}{j} + \frac{1}{n}, \quad \frac{1}{n} = \frac{p-1-\theta}{\overline{r}} + \frac{1}{r_1}.
\]

Therefore,
\begin{eqnarray*}
\|N_2\|_{L^2 L^{\frac{6}{5}}}
&\lesssim& \left\| \|u\|_{L^{r_1\theta}}^{2\theta} \cdot \|u\|_{\overline{r}}^{2p-2-2\theta} \cdot \|\nabla u\|_{3^-} \right\|_{L^2} \\
&\lesssim& \|u\|_{L^{r_1\theta}L^{\infty}}^{2\theta} \cdot \|u\|_{L^{\overline{a}}L^{\overline{r}}}^{2(p-1-\theta)} \cdot \|\nabla u\|_{4^+ 3^-},
\end{eqnarray*}
with the condition
\[
\frac{1}{2} = \frac{1}{\infty} + \frac{2(p-1-\theta)}{\overline{a}} + \frac{1}{4^+},
\]
which gives
\[
\overline{a} = \frac{8(p-1-\theta)}{1+2\varepsilon}.
\]

Since the pair \((\overline{a}, \overline{r})\) is \(\dot{H}^{s_c}\)-admissible, we conclude that
\begin{equation}\label{bar{a}}
\overline{r} = \frac{12(p-1-\theta)(p-1)}{2(p-1-\theta)(2-2b+\gamma) - (p-1)(1-2\varepsilon)}.  
\end{equation}
Observe that, to verify the condition $\bar{r}<6$, we must impose\footnote{Here, we need an extra restriction on \(p\) because the estimate involves \(3^-<3\), which is required to ensure the equivalence of Sobolev spaces. Taking the companion exponent \(3^-\) leads to the restriction stated above, while choosing a smaller exponent would give an even stricter bound on \(p\). If the potential \(V\) were absent, no such condition would be necessary, since it would suffice that the parameter remain below 6, the usual condition.
} $p<\frac{5}{2}-2b+\gamma$.
It remains to verify the integrability of \( |x|^{-b} \). From the previous identities, we have
\[
\frac{5}{6} + \frac{\gamma}{3} = \frac{1}{\beta} + \frac{1}{\sigma} = \frac{2}{j} + \frac{2}{n} + \frac{1}{3^-}.
\]
Substituting \( 3^- = \frac{3}{1+\varepsilon} \), we find $\frac{3}{j} = \frac{3}{4} + \frac{\gamma}{2} - \frac{3}{n} - \frac{\varepsilon}{2}$. Using the identity \( \frac{1}{n} = \frac{p-1-\theta}{\overline{r}} + \frac{1}{r_1} \) and the expression for \( \overline{r} \), we deduce
\[
\frac{3}{j} - b = \frac{\theta(2+\gamma-2b)}{2(p-1)} - \frac{3}{r_1}.
\]
Now assume \( A = B \), and take \( \theta r_1 = 6 \). We get
\[
\frac{3}{j} - b = \theta(1 - s_c) > 0.
\]
On the other hand, suppose \( A = B^c \), and let \( \theta r_1 = 2 \). Proceeding in the same way, we obtain
\[
\frac{3}{j} - b = -\theta s_c < 0.
\]
Thus, in both cases we conclude that \( \||x|^{-b}\|_{L^j(A)} < \infty \), completing the proof.

\item [(ii)] The proof is analogous to the estimate for \( N_2 \), replacing \( \nabla u \) with \( u \) in the expression for \( N_2 \).

 \end{itemize}

\end{proof}

\ Now that the non-linear estimates have been established, we are in a position to prove the global well-posedness result. 

\begin{proof}[\bf Proof of Propostion \ref{GWPH2}]
Let \( E > 0 \) and assume \( u_0 \in H^1(\mathbb{R}^3) \) with \( \|u_0\|_{H^1} \leq E \). Define the complete metric space
\[
B = \left\{ u \in X \,:\, \|u\|_{S(\dot{H}^{s_c})} \leq b, \ \|u\|_{S(L^2)} + \|\Lambda u\|_{S(L^2)} \leq a \right\},
\]
equipped with the distance \( d(u,v) = \|u - v\|_{S(L^2)} \). Define the operator
\[
\Phi(u)(t) := e^{-it\mathcal{H}} u_0 + i \int_0^t e^{-i(t-s)\mathcal{H}} N_\gamma(u(s)) \, ds.
\]
We verify that \( \Phi \colon B \to B \) is a contraction. The Strichartz estimates and Lemma~\ref{L:NL} yield:
\begin{align*}
\|\Phi(u)\|_{S(L^2)} &\leq \|e^{-it\mathcal{H}} u_0\|_{S(L^2)} + \left\| \int_0^t e^{-i(t-s)\mathcal{H}} N_\gamma(u(s)) \, ds \right\|_{S(L^2)} \\
&\leq c \|u_0\|_{L^2} + c \|N_\gamma(u)\|_{S'(L^2)} \\
&\leq c \|u_0\|_{L^2} + c \|u\|_{L^\infty L^2}^{\theta} \|u\|_{S(\dot{H}^{s_c})}^{2(p-1-\theta)} \|u\|_{S(L^2)}.
\end{align*}
Similarly, using the equivalence of norms, we get:
\begin{align*}
\|\Lambda \Phi(u)\|_{S(L^2)} &\leq \|e^{-it\mathcal{H}} \nabla u_0\|_{S(L^2)} + \left\| \int_0^t e^{-i(t-s)\mathcal{H}} \nabla N_\gamma(u(s)) \, ds \right\|_{S(L^2)} \\
&\leq c \|\nabla u_0\|_{L^2} + c \|u\|_{L^\infty H^1}^\theta \|u\|_{S(\dot{H}^{s_c})}^{2(p-1-\theta)} \|\nabla u\|_{L^{4^+}L^{3^-}}.
\end{align*}
Finally,
\begin{align*}
\|\Phi(u)\|_{S(\dot{H}^{s_c})} &\leq \|e^{-it\mathcal{H}} u_0\|_{S(\dot{H}^{s_c})} + \left\| \int_0^t e^{-i(t-s)\mathcal{H}} N_\gamma(u(s)) \, ds \right\|_{S(\dot{H}^{s_c})} \\
&<\delta + c \|u\|_{L^\infty H^1}^{2\theta} \|u\|_{L^a L^r}^{2p-1-2\theta}.
\end{align*}

Thus\footnote{Recalling the notation above, we write, $\|\langle\Lambda \rangle  \Phi (u)\|_{S(L^2)} :=  \| \Phi(u)\|_{S(L^2)} + \| \Lambda \Phi(u)\|_{S(L^2)}$.},
\[
\|\langle\Lambda \rangle \Phi (u)\|_{S(L^2)}  \leq c\|u_0\|_{H^1}+c a^{2\theta}b^{2(p-1-\theta)}a,
\quad \textnormal{and} \quad
\|\Phi(u)\|_{S(\dot{H}^{s_c})} < \delta + c a^{2\theta} b^{2p - 1 - 2\theta} .
\]

Combining
\begin{equation}\label{choise}
\delta = \frac{b}{2},\quad \frac{a}{2}=c\|u_0\|_{H^1}, \quad \textnormal{and} \quad b\leq \left(\frac{1}{4ca^{2\theta}}\right)^{\frac{1}{2(p-1-\theta)}},
\end{equation}
we obtain \( \Phi(u) \in B \).

\medskip

\ Now we prove that \( \Phi \) is a contraction:
\begin{align*}
\|\Phi(u) - \Phi(v)\|_{S(L^2)} 
&= \left\| \int_0^t e^{-it\mathcal{H}} \left[ N_\gamma(u(s)) - N_\gamma(v(s)) \right] \, ds \right\|_{S(L^2)} \\
&\lesssim \left\| N_\gamma(u) - N_\gamma(v) \right\|_{S'(L^2)}.
\end{align*}

Adding and subtracting the term 
\[
(I_\gamma * |x|^{-b}|u|^p)\,|x|^{-b}|v|^{p-2}v
\]
in the expression above, we write:
\begin{align*}
& (I_\gamma * |x|^{-b}|u|^p)\left( |x|^{-b}|u|^{p-2}u - |x|^{-b}|v|^{p-2}v \right) \\
&\quad + |x|^{-b}|v|^{p-2}v \left( I_\gamma * \left( |x|^{-b}|u|^p - |x|^{-b}|v|^p \right) \right).
\end{align*}

This implies:
\[
(I_\gamma * |x|^{-b}|u|^p) |x|^{-b} (|u|^{p-2} + |v|^{p-2}) |u - v|
+ |x|^{-b}|v|^{p-2}v \left( I_\gamma * \left[ |x|^{-b}(|u|^{p-1} + |v|^{p-1}) |u - v| \right] \right).
\]

Applying Lemma~\ref{L:NL}, we obtain:
\[
\|\Phi(u) - \Phi(v)\|_{S(L^2)} \leq 2c a^{2\theta}b^{2p-1-2\theta} d(u,v).
\]

Therefore, by the choice of parameters \( a \) and \( b \) in \eqref{choise}, \( \Phi \) is a contraction on \( B \), completing the proof.
\end{proof}

\section{Scattering criterion}

\vspace{0.3cm}

\ In this section, we establish the scattering criterion, following the approach of Murphy in~\cite{MurphyNonradial}, which is based on a mass evacuation condition. This criterion is in the spirit of the scattering framework introduced by Tao in~\cite{Tao}.

\begin{proposition}\label{Scattering}(Scattering criterion). Let $0<b<\frac{1+\gamma}{2}$ and $\frac{5-2b+\gamma}{3}<p<\frac{5-4b+2\gamma}{2}$. Let $V: \mathbb{R}^3 \rightarrow \mathbb{R}$ satisfy \eqref{V-potential1} and \eqref{V-potential2}. If $u\in H^{1}(\mathbb{R}^{3})$ is solution to \eqref{IGHP} defined on $[0,+\infty)$ and assume the priori bound
\begin{equation}\label{HT1}
\sup _t\|u(t)\|_{H^{1}(\mathbb{R}^{3})}=E<+\infty.
\end{equation}

There exists constants $R>0$ and $\varepsilon>0$ depending only on $E, p$ and $b$ (but not never on $u$ or $t$ ) such that if
\begin{equation}\label{HT2}
\liminf _{t \rightarrow \infty} \int_{B(0, R)}|u(t, x)|^2 d x \leq \varepsilon^2,
\end{equation}
then $u$ scatters forward in time in $H^{1}(\mathbb{R}^{3})$.
\end{proposition}

\ The proof is based on the following lemma.
\begin{lemma}[Small data scattering]\label{smalldatascattering}
Let $0 < b <\frac{1+\gamma}{2}$ and $\frac{5-2b+\gamma}{3}< p <\frac{5-4b+2\gamma}{2}$. Let $V:\mathbb{R}^3\rightarrow\mathbb{R}$ satisfy \eqref{V-potential1}, \eqref{V-potential2} and $u\in H^{1}(\mathbb{R}^{3})$ be a (possibly non-radial) solution to \eqref{IGHP}  satisfying \eqref{HT1}. If $u$ satisfies \eqref{HT2} for some $0 < \epsilon < 1$, then there exist $\gamma, T > 0$ such that 
\begin{equation}\label{norm-small}
\left\|e^{i(\cdot-T)\mathcal{H}}u(T)\right\|_{S\left(\dot{H}^{s_c}, [T,+\infty) \right)}  \lesssim\epsilon^\gamma.
\end{equation}
\end{lemma}

\begin{proof} We divide the proof into three steps.

\textbf{Step 1 – Linear estimate.}  
Fix parameters \( \alpha, \beta > 0 \). By the Strichartz estimate \eqref{CSE1}, there exists \( T_0 > \varepsilon^{-\beta} \) such that
\begin{equation}\label{T0}
\left\|e^{-it\mathcal{H}}u_0\right\|_{S\left(\dot{H}^{s_c}, [T_0,+\infty) \right)} \leq \varepsilon^{\alpha}.
\end{equation}

For any \( T \geq T_0 \), define $I_1 := [T - \varepsilon^{-\beta}, T]$ and  $ I_2 := [0, T - \varepsilon^{-\beta}].$ Let \( \eta \) be a smooth radial cutoff with \( \eta = 1 \) on \( B(0, 1/2) \) and \( \eta = 0 \) outside \( B(0,1) \), and set \( \eta_R(x) := \eta(x/R) \).

\ From Duhamel's formula,
\[
e^{-i(t-T)\mathcal{H}}u(T) = e^{it\mathcal{H}}u_0 + i(F_1 + F_2), \quad F_i := \int_{I_i} e^{-i(t-s)\mathcal{H}} N_{\gamma}(u(s)) \, ds.
\]
It remains to estimate \( F_1 \) and \( F_2 \).

\textbf{Step 2 – Estimate on recent past.}  
By \eqref{HT2}, we can choose \( T \geq T_0 \) such that
\[
\int \eta_R(x) |u(T,x)|^2 dx \lesssim \varepsilon^2.
\]
Multiplying \eqref{IGHP} by \( \eta_R \bar{u} \), taking the imaginary part and integrating by parts:
\begin{align}
\left| \frac{d}{dt} \int \eta_R |u|^2 dx \right|
&= \left| 2 \int \nabla \eta_R \cdot \operatorname{Im}(\bar{u} \nabla u) dx \right| \\
&\lesssim \|\nabla \eta_R\|_{L^\infty} \|u\|_{L^2} \|\nabla u\|_{L^2} \lesssim \frac{1}{R}.
\end{align}
Then for \( t \leq T \),
\begin{align}
\int \eta_R |u(t)|^2 dx &\leq \int \eta_R |u(T)|^2 dx + \frac{T - t}{R} \lesssim \varepsilon^2 + \frac{\varepsilon^{-\beta}}{R}.
\end{align}
Thus, for \( R > \varepsilon^{-(\beta + 2)} \), $\| \eta_R u \|_{L^\infty_{I_1} L^2_x} \lesssim \varepsilon^2.$

\ Now, using the pair \( (\widetilde{a}, r) \in A_{-s_c} \) from Lemma \ref{L:NL}, and applying Hölder and Sobolev for \( t \in I_1 \),
\begin{equation}\label{recent_past_ball1}
\| \eta_R N_\gamma(u) \|_{L_x^{r'}} \lesssim \|u(t)\|_{H^1}^{2\theta} \|u(t)\|_{L^r}^{2(p-1-\theta)} \|\eta_R u(t)\|_{L^r} \lesssim \|\eta_R u(t)\|_{L^r}.
\end{equation}
Let \( \hat{\theta} \in (0,1) \) solve \( \frac{1}{r} = \frac{\hat{\theta}}{2} + \frac{1 - \hat{\theta}}{p^*} \). Then,
\begin{equation}\label{recent_past_ball2}
\|\eta_R u(t)\|_{L^r} \leq \|u(t)\|_{L^{p^*}}^{1 - \hat{\theta}} \|\eta_R u(t)\|_{L^2}^{\hat{\theta}} \lesssim \varepsilon^{\hat{\theta}},
\end{equation}
uniformly in \( t \in I_1 \).

\ The exterior region is estimated using the decay of the nonlinearity, avoiding radial symmetry
\begin{align}\label{recent_past_out}
\| (1 - \eta_R) N_\gamma(u) \|_{L^{r'}} 
&\lesssim \| (I_\gamma * |x|^{-b} |u|^p) |x|^{-b} |u|^{p-2} u \|_{L_{\{|x| > R/2\}}^{r'}} \nonumber \\
&\lesssim \| |x|^{-b} \|_{L^{r_1}_{\{|x| > R/2\}}} \|u(t)\|^{2p - 1}_{H^1} \lesssim \varepsilon^{\hat{\theta}},
\end{align}
where \( br_1 > 3 \), \( \theta r_2 \in (2, 3p/(2-b)) \).

\ Combining \eqref{CSE2}, \eqref{recent_past_ball1}, \eqref{recent_past_ball2}, and \eqref{recent_past_out}:
\begin{align*}
\left\| \int_{I_1} e^{-i(t-s)\mathcal{H}} N_\gamma(u(s)) \, ds \right\|_{S(\dot{H}^{s_c}, [T, \infty))} 
&\lesssim \| \eta_R N_\gamma \|_{L^{\widetilde{a}'}_{I_1} L^{r'}_x} + \| (1 - \eta_R) N_\gamma \|_{L^{\widetilde{a}'}_{I_1} L^{r'}_x} \\
&\lesssim |I_1|^{1/\widetilde{a}'} \varepsilon^{\hat{\theta}} = \varepsilon^{\hat{\theta} - \beta / \widetilde{a}'} = \varepsilon^{\hat{\theta}/2},
\end{align*}
where \( \beta := \widetilde{a}' \hat{\theta} / 2 \).
 
\textbf{Step 3 – Estimate on the distant past.}
We now estimate \( F_2 \). Let \( (a, r) \in A_{s_c} \), and define
\begin{equation}
\frac{1}{c} = \left(\frac{1}{1-s_c}\right)\left[\frac{1}{a} - \delta s_c\right],
\end{equation}
and
\begin{equation}
\frac{1}{d} = \left(\frac{1}{1-s_c}\right)\left[\frac{1}{r} - s_c\left(\frac{3 - 2 - 4\delta}{6}\right)\right],
\end{equation}
for some small \( \delta > 0 \). It is easy to verify that the pair \( (c, d) \) is \( \Lambda_0 \)-admissible. By interpolation
\[
\|F_2\|_{L^a_t L^r_x([T, +\infty))} \leq 
\|F_2\|_{L^c_t L^d_x([T, +\infty))}^{1-s_c}
\|F_2\|_{L^{1/\delta}_t L^{6/(3 - 2 - 4\delta)}_x([T, +\infty))}^{s_c}.
\]

We now rewrite \( F_2 \) using Duhamel’s formula:
\[
F_2 = e^{-it\mathcal{H}}\left[e^{-i(-T + \varepsilon^{-\beta})\mathcal{H}}u(T - \varepsilon^{-\beta}) - u(0)\right].
\]
Applying the Strichartz estimate \eqref{CSE1}, we deduce
\begin{align}
\|F_2\|_{L^a_t L^r_x([T, +\infty))} 
&\leq 
\left\| e^{-it\mathcal{H}}\left[e^{-i(-T + \varepsilon^{-\beta})\mathcal{H}}u(T - \varepsilon^{-\beta}) - u(0)\right] \right\|_{L^c_t L^d_x}^{1-s_c}
\|F_2\|_{L^{1/\delta}_t L^{6/(3 - 2 - 4\delta)}_x}^{s_c} \\
&\lesssim 
\left(\|u\|_{L^\infty_t L^2_x} \right)^{1 - s_c}
\|F_2\|_{L^{1/\delta}_t L^{6/(3 - 2 - 4\delta)}_x}^{s_c}
\lesssim \varepsilon^{\beta \delta s_c}.
\end{align}

Next, using the decay estimate for the free Schrödinger group,
\[
\|e^{-it\mathcal{H}} f\|_{L^r_x} \lesssim \frac{1}{|t|^{\frac{3}{2}(\frac{1}{r'} - \frac{1}{r})}} \|f\|_{L^{r'}_x}, \quad \forall r \geq 2,
\]
combined with \eqref{DE1}, we obtain
\begin{align}
\|F_2\|_{L^{1/\delta}_t L^{6/(3 - 2 - 4\delta)}_x([T, +\infty))}
&\lesssim \left\| \int_{I_2} |\cdot - s|^{-(1 + 2\delta)} \|N_\gamma(u(s))\|_{L^{6/(3 + 2 + 4\delta)}_x} ds \right\|_{L^{1/\delta}_t} \\
&\lesssim \|u\|_{L^\infty_t H^1_x}^p \left\| (\cdot - T + \varepsilon^{-\mu})^{-2\delta} \right\|_{L^{1/\delta}_t([T, +\infty))} \lesssim \varepsilon^{\beta \delta}.
\end{align}

Finally, define $
\gamma := \min\left\{ \frac{\hat{\theta}}{2}, \beta \delta s_c \right\},$
and recall that
\[
e^{-i(t - T)\mathcal{H}} u(T) = e^{-it\mathcal{H}} u_0 + i F_1 + i F_2.
\]
Combining the estimates for \( F_1 \) and \( F_2 \), we conclude that
\begin{equation}
\left\|e^{-i(t - T)\mathcal{H}} u(T) - e^{-it\mathcal{H}} u_0 \right\|_{S(\dot{H}^{s_c}, [T, +\infty))} \leq C \varepsilon^\gamma,
\end{equation}
which completes the proof.
\end{proof}

\begin{proof}[\bf Proof of Proposition \ref{Scattering}]
We only need to prove that \( \| u \|_{S(\dot{H}^{s_c}, [T, +\infty))} < \infty \), since we already have \( u \in L^\infty_t H^1_x \); in this case, the finiteness of both norms implies scattering. In fact, choose \( \varepsilon > 0 \) sufficiently small so that, by Lemma~\ref{smalldatascattering}, we obtain
\[
\left\|e^{-i(\cdot)\mathcal{H}}u(T)\right\|_{S\left(\dot{H}^{s_c}, [0, +\infty)\right)} 
= \left\|e^{-i(\cdot - T)\mathcal{H}} u(T)\right\|_{S\left(\dot{H}^{s_c}, [T, +\infty)\right)} 
\leq C \varepsilon^\gamma \leq \delta_{sd},
\]
so that the \( S(\dot{H}^{s_c}) \)-norm is bounded by the global small data theory, as stated in Proposition~\eqref{GWPH2}.
\end{proof}
\begin{remark}\label{important}
As already noted in the proof of the previous proposition, scattering follows once both norms are finite. The argument is standard. For \( \eta > 0 \), write $[T, +\infty) = \bigcup_{j=1}^{N} I_j,$ where the intervals \( I_j \) are chosen so that
$\|u\|_{L^a_t(I_j; L^r_x)} < \eta.$ By the Strichartz estimates, it follows that
$\|\nabla u\|_{S(L^2, [T, +\infty))} < \infty.$ Define
\[
\phi^+ := e^{-it\mathcal{H}}u(T) + i \int_T^{+\infty} e^{-i(t-s)\mathcal{H}} N(u)(s) \, ds.
\]
Then, by the nonlinear estimate (Lemma \ref{L:NL} (i-ii))
\[
\|u(t) - e^{-it\mathcal{H}} \phi^+\|_{H^1} \lesssim \|u\|_{L^a_t([t, +\infty); L^r_x)}^{2p - 2} \|\nabla u\|_{L^{4+}_t([t, +\infty); L^{3-}_x)} \to 0, \quad \text{as}\;\; t \to +\infty.
\]

\ This conclusion relies on the nonlinear estimates from Lemma~\ref{L:NL} and the use of an admissible pair compatible with the \( S(\dot{H}^{s_c}) \) framework, ensuring that scattering is rigorously established in the critical norm.

\end{remark}

\section{Variational Analysis}

\vspace{0.3cm}

We recall key properties of the ground state \( Q \), which is a positive, radial, and decreasing solution to
\begin{equation*}
\Delta Q - Q + (I_{\gamma} * |x|^{-b} |Q|^{p}) |x|^{-b} |Q|^{p-2} Q = 0.
\end{equation*}

The function \( Q \) achieves the sharp constant in the Gagliardo–Nirenberg inequality:
\begin{equation}
\int_{\mathbb{R}^N} (I_{\gamma} * |x|^{-b} |u|^{p}) |x|^{-b} |u|^{p} \, dx \leq C_0 \|u\|_{L^2}^{A} \|\nabla u\|_{L^2}^{B},
\end{equation}
with exponents $B = 3p - 3 + \gamma - 2b = 2(p - 1)s_c + 2$ and $ A = 2(p - 1)(1 - s_c)$. Moreover, \( Q \) satisfies
\[
\|\nabla Q\|_{L^2}^2 = \frac{B}{A} \|Q\|_{L^2}^2,
\]
and
\begin{equation}
\int_{\mathbb{R}^N} (I_{\gamma} * |x|^{-b} |Q|^{p}) |x|^{-b} |Q|^{p} \, dx = \frac{2p}{B} \|\nabla Q\|_{L^2}^2.
\end{equation}
Thus,
\begin{equation}
\frac{2p}{B} \leq C_0 \|u\|_{L^2}^{2(p-1)(1 - s_c)} \|\nabla u\|_{L^2}^{2(p-1)s_c}.
\end{equation}

The Pohozaev identities yield
\[
E_0(Q) = \frac{B - 2}{2B} \|\nabla Q\|_{L^2}^2 = \frac{B - 2}{2A} \|Q\|_{L^2}^2,
\]
where \( E_0(Q) = \frac{1}{2} \|\nabla Q\|_{L^2}^2 - \frac{1}{2p} P(Q) \). From this,
\begin{align*}
\frac{1}{2p} P(Q) &= \left( \frac{1}{2} - \frac{B - 2}{2B} \right) \|\nabla Q\|_{L^2}^2 = \frac{1}{B} \|\nabla Q\|_{L^2}^2, \\
\Rightarrow \quad \|\nabla Q\|_{L^2} &= \left( \frac{B}{2p} P(Q) \right)^{1/2}.
\end{align*}

\subsection{Coercivity}

Let $\chi:\mathbb{R}^{3} \to [0,+\infty)$ be a smooth cutoff function satisfying \( 0 \leq \chi(x) \leq 1 \), and
\begin{equation*}
\chi(x)=\begin{cases}
    1, & |x| < \frac{1}{2}, \\
    0, & |x| > 1.
\end{cases}
\end{equation*}
For any \( R > 0 \), define the rescaled cutoff \( \chi_R(x) := \chi\left(\frac{x}{R}\right) \). In particular, it is easy to verify that
\[
\|\chi_R u\|_{L^2} \leq \|u\|_{L^2} \quad \text{and} \quad P(\chi_R u) \leq P(u).
\]

\begin{lemma}[Coercivity]\label{lem_coerc}
Let \( u \in H^1(\mathbb{R}^3) \). Suppose there exists \( \delta \in (0,1) \) such that
\begin{equation}\label{lemmacoercivity}
    P(u(t)) M(u(t))^{\sigma_c} \leq (1 - \delta) P(Q) M(Q)^{\sigma_c}.
\end{equation}
Then there exists a constant \( \delta' > 0 \) such that
\begin{equation}
    \int_{\mathbb{R}^3} |\nabla u|^2 \, dx - \frac{B}{2p} P(u) \geq \delta' P(u).
\end{equation}
In particular, for all \( R > 0 \),
\begin{equation}
    \int_{\mathbb{R}^3} |\nabla (\chi_R u)|^2 \, dx - \frac{B}{2p} P(\chi_R u) \geq \delta' P(\chi_R u).
\end{equation}
\end{lemma}
\begin{proof}
Since the ground state \( Q \) optimizes the sharp Gagliardo–Nirenberg inequality, we have
\begin{equation}\label{eq:GN}
P(u) \leq C_{\mathrm{op}} \|u\|_{L^2}^{A} \|\nabla u\|_{L^2}^{B}.
\end{equation}
Rewriting this, we obtain
\begin{align*}
1 &\leq C_{\mathrm{op}} P(u)^{-1} \|u\|_{L^2}^{A} \|\nabla u\|_{L^2}^{B}, \\
P(u)^{\frac{B}{2}} &\leq C_{\mathrm{op}} P(u)^{\frac{B}{2}-1} M(u)^{\frac{A}{2}} \|\nabla u\|_{L^2}^{B}.
\end{align*}
Note that \( \frac{B}{2} - 1 = \frac{A}{2\sigma_c} \), thus
\begin{equation}\label{eq:Pu}
P(u)^{\frac{B}{2}} \leq C_{\mathrm{op}} (P(u) M(u)^{\sigma_c})^{\frac{B}{2} - 1} \|\nabla u\|_{L^2}^{B}.
\end{equation}

Now, compute the optimal constant \( C_{\mathrm{op}} \) using \( \|\nabla Q\|_{L^2}^2 = \frac{B}{2p} P(Q) \),
\begin{align*}
C_{\mathrm{op}} &= \frac{P(Q)}{\|Q\|_{L^2}^A \|\nabla Q\|_{L^2}^B}
= \frac{P(Q)}{M(Q)^{\frac{A}{2}} \left( \frac{B}{2p} P(Q) \right)^{\frac{B}{2}}}= \frac{(2p)^{\frac{B}{2}}}{B^{\frac{B}{2}} (P(Q) M(Q)^{\sigma_c})^{\frac{B}{2}-1}}.
\end{align*}

Substituting into \eqref{eq:Pu}, we get
\begin{equation}\label{eq:key}
P(u)^{\frac{B}{2}} \leq \left( \frac{P(u) M(u)^{\sigma_c}}{P(Q) M(Q)^{\sigma_c}} \right)^{\frac{B}{2} - 1} \left( \frac{2p}{B} \|\nabla u\|_{L^2}^2 \right)^{\frac{B}{2}}.
\end{equation}

Under the assumption \( P(u) M(u)^{\sigma_c} \leq (1-\delta) P(Q) M(Q)^{\sigma_c} \), inequality \eqref{eq:key} implies
\[
P(u)^{\frac{B}{2}} \leq (1-\delta)^{\frac{B}{2} - 1} \left( \frac{2p}{B} \|\nabla u\|_{L^2}^2 \right)^{\frac{B}{2}},
\]
so that
\begin{equation}\label{eq:coercivity}
\|\nabla u\|_{L^2}^2 \geq \frac{B}{2p} \cdot \frac{1}{(1-\delta)^{\frac{B-2}{B}}} P(u).
\end{equation}

Subtracting \( \frac{B}{2p}P(u) \) from both sides in \eqref{eq:coercivity}, we obtain
\begin{align*}
\|\nabla u\|_{L^2}^2 - \frac{B}{2p} P(u)
&\geq \frac{B}{2p} \left( \frac{1}{(1-\delta)^{\frac{B-2}{B}}} - 1 \right) P(u) \\
&= \delta' P(u),
\end{align*}
where $\delta' := \frac{B}{2p} \left( \frac{1}{(1-\delta)^{\frac{B-2}{B}}} - 1 \right)$.

\ The second part of the lemma follows immediately, since \( \|\chi_R u\|_{L^2} \leq \|u\|_{L^2} \) and \( P(\chi_R u) \leq P(u) \), concluding the proof.

\end{proof}

\ To conclude this section, we apply variational analysis to derive a consequence of our main result.
 
\begin{proof}[\bf Proof of Corollary \eqref{maincorollary}]
We now prove a consequence of Theorem~\ref{mainresulttheorem}: under the standard mass and energy conditions, scattering follows.

\ First, we show that \( u \) is global. Then we prove
\[
\sup_{t \in \mathbb{R}} P(u(t))M(u(t))^{\sigma_c} < P(Q)M(Q)^{\sigma_c},
\]
which ensures scattering by Theorem~\ref{mainresulttheorem}.

\ From assumption \eqref{1}, we have
\[
M(u)^{\sigma_c}E(u_0) < (1 - \delta) M(Q)^{\sigma_c}E_0(Q), \quad \delta > 0.
\]
Using this and the sharp Gagliardo–Nirenberg inequality:
\begin{align*}
M(u)^{\sigma_c}E(u) 
&= \|u\|_{L^2}^{2\sigma_c}\left(\tfrac{1}{2}\|\Lambda u(t)\|_{L^2}^2 - \tfrac{1}{2p}P(u(t))\right) \\
&\geq \tfrac{1}{2}(\|\Lambda u(t)\|\|u(t)\|^{\sigma_c})^2 - \tfrac{C_{op}}{2p}\|u(t)\|^{A+2\sigma_c} \|\Lambda u(t)\|^B \\
&= \tfrac{1}{2}(\|\Lambda u(t)\|\|u(t)\|^{\sigma_c})^2 - \tfrac{C_{op}}{2p}(\|u(t)\|^{\sigma_c} \|\Lambda u(t)\|)^B,
\end{align*}
where we used \( A + 2\sigma_c = B\sigma_c \). Define \( g(x) := \tfrac{1}{2}x^2 - \tfrac{C_{op}}{2p}x^B \). Then,
\[
(1 - \delta) M(Q)^{\sigma_c}E_0(Q) > g(\|\Lambda u(t)\|\|u(t)\|^{\sigma_c}).
\]
Moreover, 
\[
M(Q)^{\sigma_c}E_0(Q) = g(\|\nabla Q\|\|Q\|^{\sigma_c}).
\]
Hence,
\[
g(\|\Lambda u(t)\|\|u(t)\|^{\sigma_c}) < g(\|\nabla Q\|\|Q\|^{\sigma_c}).
\]
By continuity of \( g \) and assumption \eqref{2}, we get
\[
\|\Lambda u(t)\|\|u(t)\|^{\sigma_c} < \|\nabla Q\|\|Q\|^{\sigma_c},
\]
implying \( u \) is global. This completes the first part.

\ On other hand,
\[
1 - \delta > \frac{\frac{1}{2}(\|\Lambda u(t)\|\|u(t)\|^{\sigma_c})^2}{E_0(Q)M(Q)^{\sigma_c}} - \frac{C_{op}}{2p} \frac{(\|\Lambda u(t)\|\|u(t)\|^{\sigma_c})^B}{E_0(Q)M(Q)^{\sigma_c}}.
\]
Using Pohozaev identity:
\[
E_0(Q)M(Q)^{\sigma_c} = \frac{B-2}{2B}(\|\nabla Q\|\|Q\|^{\sigma_c})^2\quad C_{op} = \frac{P(Q)}{\|Q\|^A\|\nabla Q\|^B} = \frac{2p}{B\|Q\|^A\|\nabla Q\|^{B-2}},
\]
we obtain
\[
1 - \delta > \frac{B}{B-2}\left(\frac{\|\Lambda u\|\|u\|^{\sigma_c}}{\|\nabla Q\|\|Q\|^{\sigma_c}}\right)^2 - \frac{2}{B-2} \left(\frac{\|\Lambda u\|\|u\|^{\sigma_c}}{\|\nabla Q\|\|Q\|^{\sigma_c}}\right)^B.
\]
Set \( y(t) := \frac{\|\Lambda u\|\|u\|^{\sigma_c}}{\|\nabla Q\|\|Q\|^{\sigma_c}} \). Since \( y(t) < 1 \), define
\[
f(y) = \frac{B}{B-2}y^2 - \frac{2}{B-2}y^B,
\]
with \( f \) increasing, \( f(0) = 0 \), and \( f(1) = 1 \). Hence, there exists \( \rho > 0 \) such that
\begin{equation}\label{eq:critical_gap}
\|\Lambda u\|^2\|u\|^{\sigma_c} < (1 - \rho)\|\nabla Q\|\|Q\|^{\sigma_c}.
\end{equation}

Next, we use the inequality
\begin{equation}\label{eq:nonlinear_prod}
M(u)^{\sigma_c} P(u) \leq C_{op} \|u\|^{A + 2\sigma_c} \|\nabla u\|^B = C_{op} (\|u\|^{\sigma_c} \|\nabla u\|)^B.
\end{equation}
From
\[
C_{op} = \frac{P(Q)}{\|Q\|^A \|\nabla Q\|^B} = \frac{P(Q) M(Q)^{\sigma_c}}{\|Q\|^{A + 2\sigma_c} \|\nabla Q\|^B},
\]
and since \( B\sigma_c = A + 2\sigma_c \), we find
\[
\frac{P(u) M(u)^{\sigma_c}}{P(Q) M(Q)^{\sigma_c}} \leq \left( \frac{\|\nabla u\| \|u\|^{\sigma_c}}{\|\nabla Q\| \|Q\|^{\sigma_c}} \right)^B.
\]
By \eqref{eq:critical_gap},
\[
P(u) M(u)^{\sigma_c} < (1 - \rho)^B P(Q) M(Q)^{\sigma_c}.
\]
Therefore, \( u \) scatters in \( H^1(\mathbb{R}^3) \).

\end{proof}

\section{Proof of Theorem \ref{mainresulttheorem}}

\ In this section, we establish a Virial–Morawetz estimate that yields the mass evacuation condition. In particular, we conclude that \( u \) is global, uniformly bounded in \( H^1(\mathbb{R}^3) \), and satisfies Proposition~\ref{Scattering}, thus completing the proof of our main result.

\ Let $R \gg 1$ to be chosen later. We take $a(x)$ to be a radial function satisfying
$$
a(x)= \begin{cases}|x|^2 ; & |x| \leq R \\ 3R|x| ; & |x|>2 R,\end{cases}
$$
and when $R<|x| \leq 2 R$, there holds
$$
\partial_r a \geq 0,\qquad \partial_{r r} a \geq 0 \quad \text { and } \quad\left|\partial^\alpha a\right| \lesssim R|x|^{-|\alpha|+1} .
$$

Here $\partial_r$ denotes the radial derivative. Under these conditions, the matrix $\left(a_{j k}\right)$ is non-negative. It is easy to verify that
$$
\begin{cases}a_{j k}=2 \delta_{j k}, \quad \Delta a=6, \quad \Delta^{2} a=0, & |x| \leq R \\ a_{j k}=\frac{3 R}{|x|}\left[\delta_{j k}-\frac{x_j x_k}{|x|^2}\right], \quad \Delta a=\frac{6R}{|x|}, \quad \Delta^{2} a=0, & |x|>2 R\end{cases}
$$
\subsection{Virial-Morawetz estimate}

\ We first derive the Virial–Morawetz identity. The proof is standard and will be omitted.

\begin{lemma}\label{morawetzidentity}{(Virial-Morawetz identity)} Let $a:\mathbb{R}^{3}\rightarrow\mathbb{R}$ a smooth weight. If $\|\nabla a\|_{L^{\infty}}$ is bounded, define
\begin{equation}\label{identity}
M_{a}(t) = 2Im\int_{\mathbb{R}^{3}}\bar{u}\nabla{u}\cdot \nabla{a}\ dx
\end{equation}
Then, we have
\begin{subequations}\label{eq:parte4}
\begin{align}
\frac{d}{dt} M_a(t) &= \int_{\mathbb{R}^3} (-\Delta^2 a) |u|^2\, dx 
+ 4 \int_{\mathbb{R}^3} a_{jk} \operatorname{Re}(\partial_j \bar{u} \partial_k u)\, dx 
\label{eq:parte4a} \\
&\quad 
- \left(2 - \frac{4}{p} \right) \int_{\mathbb{R}^3} \Delta a\, |u|^p (I_\gamma * |x|^{-b} |u|^p)\, |x|^{-b} dx 
\label{eq:parte4b} \\
&\quad 
+ \frac{4b}{p} \int_{\mathbb{R}^3} \nabla a \cdot \frac{x}{|x|^{2+b}} (I_\gamma * |x|^{-b} |u|^p)\, |u|^p\, dx 
\label{eq:parte4c} \\
&\quad 
- 2 \int_{\mathbb{R}^3} |u|^2\, \nabla V \cdot \nabla a(x)\, dx 
\label{eq:parte4d} \\
&\quad 
- M 
\iint_{\mathbb{R}^3 \times \mathbb{R}^3}
(\nabla a(x) - \nabla a(y)) \cdot \frac{x - y}{|x - y|^{5 - \gamma}} 
|x|^{-b} |u(x)|^p |y|^{-b} |u(y)|^p\, dy\, dx,
\label{eq:parte4e}
\end{align}
where $M=\frac{2\mathcal{K}(3-\gamma)}{p}$.
\end{subequations}

\end{lemma}
\begin{remark}
Define the domain
\[
\Omega := \{x \in \mathbb{R}^3 : |x| \leq R\} \times \{y \in \mathbb{R}^3 : |y| \leq R\} \subset \mathbb{R}^3 \times \mathbb{R}^3, 
\quad \Omega^c := \mathbb{R}^3 \times \mathbb{R}^3 \setminus \Omega.
\]
Note that on \( \Omega \), we have \( \nabla a(x) - \nabla a(y) = 2(x - y) \), and we can decompose the nonlocal term from ~\eqref{eq:parte4e} as
\begin{align*}
\frac{p}{2\mathcal{K}(3-\gamma)}\eqref{eq:parte4e} 
&= \iint_{\Omega} \frac{2 |x|^{-b} |u(x)|^p |y|^{-b} |u(y)|^p}{|x - y|^{3 - \gamma}}\, dy\, dx \\
&\quad + \iint_{\Omega^c} 
(\nabla a(x) - \nabla a(y)) \cdot \frac{x - y}{|x - y|^{5 - \gamma}} 
|x|^{-b} |u(x)|^p |y|^{-b} |u(y)|^p\, dy\, dx\\
&:=I_1+I_2.
\end{align*}

For the second integral, from the definition of \( a(x) \), we have 
\[
|\partial^\beta a(x)| \lesssim R^{2 - |\beta|}, \quad \text{when } |x - y| < R,
\]
and
\[
\left| (\nabla a(x) - \nabla a(y)) \cdot \frac{x - y}{|x - y|^2} \right| \lesssim \| D^2 a \|_{L^\infty}.
\]
Similarly, when \( |x - y| \geq R \), it holds that
\[
\left| (\nabla a(x) - \nabla a(y)) \cdot \frac{x - y}{|x - y|^2} \right| 
\lesssim \frac{\| D^2 a \|_{L^\infty}}{R}.
\]

Moreover, by symmetry and suitable estimates,
\[
I_2
\leq \frac{2}{\mathcal{K}} \int_{|x| > 2R} 
(I_\gamma * |x|^{-b} |u|^p) |x|^{-b} |u|^p\, dx.
\]
\end{remark}

\ We now establish the next proposition.

\begin{proposition}\label{prop5.3} Let $0<b<\frac{1+\gamma}{2}$ and $\frac{5-2b+\gamma}{3}< p<\frac{5-4b+2\gamma}{2}$. Let $V: \mathbb{R}^3 \rightarrow \mathbb{R}$ satisfy \eqref{V-potential1}, $V\geq0$, $x\cdot\nabla V\leq 0$ and $x\cdot\nabla V\in L^{r}$, where $r\in[\frac{3}{2},\infty)$.  Let $u$ be a $H^1$-solution of \eqref{IGHP}
 satisfying \eqref{newcondition}. Then for any $T>0$ and $R(\delta, M(u),Q)$ sufficiently
large
\begin{equation}\label{pro5.3a}
\frac{1}{T}\int_{0}^{T}\int_{|x|<R}P(u(t))dxdt\lesssim \frac{R}{T}+\mathcal{O}\left(\frac{1}{R^{2}}+\frac{1}{R^{b}}+\mathcal{O}_R\right).
\end{equation}
In particular, there exists a sequence of times $t_{n},R_{n}\rightarrow\infty$ such that 
\begin{equation}\label{prop5.3b}
\lim_{n\rightarrow\infty}\int_{|x|<R_{n}}P(\chi_{R}u(t_n))dx = 0.
\end{equation}
\end{proposition}
\begin{proof}
Let $\delta'=\delta'(u_{0},Q)$ and $R_{0}=R_{0}(\delta,u_{0})$. Let $a$ as a lemma \eqref{morawetzidentity}. By Cauchy-Schwartz and Lemma \eqref{morawetzidentity}, we obtain\footnote{Note that, by combining condition \eqref{newcondition} with the conservation of energy, we deduce that $\|\Lambda u(t)\|_{L^{2}}$ is bounded.
} 
\begin{eqnarray*}
\sup_{t\in\mathbb{R}}|M_{a}(t)|= \sup_{t\in\mathbb{R}}|2Im\int_{\mathbb{R}^{3}}\bar{u}\nabla{u}\cdot \nabla{a}\ dx| &\leq&\|\nabla a(x)\|_{L^{\infty}}\|u(t)\|_{L^{2}}\|\Lambda u(t)\|_{L^{2}}\lesssim R
\end{eqnarray*}
For all $t\in\mathbb{R}$, by Virial-Morawetz identity we have
\begin{subequations}
\begin{align}
\frac{d}{d t} M_{a}(t)&=8\left(\int_{|x|<R}|\nabla u(t)|^2 d x-\frac{B}{2p} \int_{|x|<R}\left(I_\gamma *|x|^{-b}|u|^p\right)|u|^p|x|^{-b} d x\right)\label{5.19}\\
&\quad-2\int_{|x|<R}|u|^{2}\nabla V \cdot\nabla a (x)dx\label{5.20}\\
&\quad-\frac{4\mathcal{K}(3-\gamma)}{p}\int_{|x|<R}\int_{|x|<R}\frac{|y|^{-b}|u(y)|^p|x|^{-b}|u(x)|^p}{|x-y|^{3-\gamma}} d y d x\label{5.21}\\
&\quad+4\int_{R<|x|<2 R}a_{j k} \operatorname{Re}\left(\partial_j \bar{u} \partial_k u\right)+\mathcal{O}\left(\frac{R}{|x|^3}|u|^2\right) d x\label{5.22}\\
&\quad+\int_{|x|>2 R}\frac{9 R}{|x|}|\not\nabla u|^2 d x +\mathcal{O}\left(\int_{|x|>2R}(I_{\gamma}*|x|^{-b}|u|^{p})|x|^{-b}|u|^{p}dx\right)\label{5.23}\\
&\quad-2\int_{|x|>2R}|u|^{2}\nabla V \cdot\nabla a (x)dx\label{5.24}\\
&\quad-2\int_{R<|x|<2R}|u|^{2}\nabla V \cdot\nabla a(x)dx\label{5.25},
\end{align}
\end{subequations}
where $\not\nabla=\nabla-\frac{x}{|x|^2}\,(x \cdot \nabla)$ denotes the angular part of the derivative. 
In the first term of \eqref{5.23}, the angular derivative term is nonnegative. 
In \eqref{5.22}, the first term is also nonnegative, and the second one is bounded by $R^{-2}$. 
For the second term of \eqref{5.23}, by the Hardy--Littlewood--Sobolev inequality we get
$$
\mathcal{O}\!\left(\int_{|x|>2R}(I_{\gamma}*|x|^{-b}|u|^{p})\,|x|^{-b}|u|^{p}\,dx\right)
\lesssim R^{-b}\,\|\,|x|^{-b}\,\|_{L^{r_1}(A)}\,\|u\|_{L^r}^{2p}
\lesssim R^{-b},
$$
where the exponents satisfy
$$
1+\frac{\gamma}{3}=\frac{1}{r_1}+\frac{1}{r}.
$$
Since $p \geq 2$, for each $A \in \{B_1(0), B_1^c(0)\}$ there exists $r \in [2,6]$ such that $|x|^{-b}\in L^{r_1}(A)$.

In \eqref{5.24}, the term follows from the fact that $x \cdot \nabla V \leq 0$, and is therefore nonnegative. 
Using that $\left|\nabla a \cdot \nabla V\right|\lesssim 2|x \cdot \nabla V|$ and $x \cdot \nabla V \in L^{r}$, with $r\in[\tfrac{3}{2},\infty)$, 
we know from Sobolev embedding and interpolation that
\begin{equation}
\|u\|_{L^{q}(\mathbb{R}^{3})}\leq \|u\|_{L^{3}(\mathbb{R}^{3})}^{\sigma}\,
\|u\|_{L^{6}(\mathbb{R}^{3})}^{1-\sigma}
\leq \|u\|_{H^{1}(\mathbb{R}^{3})}, 
\qquad 3\leq q\leq 6.
\end{equation}
For all $r\geq \tfrac{3}{2}$, we note that $2\leq 2r'\leq 6$. Therefore,
\begin{eqnarray*}
\left|\int_{|x|>2R} |u(t)|^{2}\,\nabla V \cdot \nabla a \, dx\right|
&\lesssim& \int_{|x|>2R} |u(t)|^{2}\,|\nabla a \cdot \nabla V|\, dx \\
&\lesssim& \int_{\mathbb{R}^{3}} |u(t)|^2\,|x \cdot \nabla V|\, dx \\
&\lesssim& \|x \cdot \nabla V\|_{L^{r}(|x|>2R)} \,\|u(t)\|_{L^{2r'}}^2 \\
&\lesssim& \|x \cdot \nabla V\|_{L^{r}(|x|>2R)} \,\|u(t)\|_{H^{1}(\mathbb{R}^{3})} \\
&=& \mathcal{O}_R.
\end{eqnarray*}

In the \eqref{5.25}, using the same idea above and know that $\left|\partial^\alpha a\right| \lesssim R|x|^{-|\alpha|+1}$, we have that 
\begin{equation}
\left|\int_{R<|x|<2R}\nabla a\nabla V |u(t)|^{2} dx\right|\lesssim\mathcal{O}_{R}.
\end{equation}

The term in \eqref{5.20} is analogous to the terms in \eqref{5.24} and \eqref{5.25}.  
Thus, we have 
\begin{eqnarray*}
\frac{d}{dt} M_{a}(t)
&\geq& 8\left(\int_{|x|<R}|\nabla u(t)|^2 \, dx
-\frac{B}{2p}\int_{|x|<R}\left(I_\gamma * |x|^{-b}|u|^p\right)|u|^p|x|^{-b}\, dx\right) \\
&+& \mathcal{O}\!\left(\frac{1}{R^{2}}+\frac{1}{R^{b}}+\mathcal{O}_R\right).
\end{eqnarray*}

Let $\chi_R$ be as in Lemma~\eqref{lem_coerc}. We compute
\[
\begin{aligned}
\int\left|\nabla(\chi_R u(t))\right|^2 dx 
&= \int \chi_R^2 |\nabla u(t)|^2 dx - \int \chi_R \Delta(\chi_R)|u(t)|^2 dx \\
&= \int_{|x|\leq R}|\nabla u(t)|^2 dx 
   - \int_{R\leq |x|\leq 2R}(1-\chi_R^2)|\nabla u(t)|^2 dx 
   - \int \chi_R \Delta(\chi_R)|u(t)|^2 dx.
\end{aligned}
\]

Since $0\leq \chi_R \leq 1$ and $\|\Delta(\chi_R)\|_{L^{\infty}} \leq R^{-2}$, it follows that
\begin{eqnarray*}
\eqref{5.19}
&=& 8\Bigg(\int_{\mathbb{R}^3}\left|\nabla(\chi_R u)\right|^2 dx
- \frac{B}{2p}\int_{\mathbb{R}^3}\left(I_\gamma * |x|^{-b}|\chi_R u|^p\right)|x|^{-b}|\chi_R u|^p dx\Bigg) \\
&& \quad + \int_{\mathbb{R}^3} \mathcal{O}\!\left(\frac{|u|^2}{R^2}\right) dx + \frac{1}{R^{b}}.
\end{eqnarray*}

By Lemma~\eqref{lem_coerc}, we have
\[
\int_{\mathbb{R}^3}|\nabla (\chi_{R}u)|^2 dx-\frac{B}{2p}P(\chi_{R}u)\geq \delta' P(\chi_{R}u).
\]

Therefore,
\begin{eqnarray*}
\eqref{5.19}
&\geq& 8\delta'\int_{\mathbb{R}^3}P(\chi_{R}u)\, dx
+ \int_{|x|<R}\mathcal{O}\!\left(\frac{|u|^2}{R^2}\right) dx + \frac{1}{R^{b}}.
\end{eqnarray*}

Hence,
\begin{eqnarray}
\frac{d}{dt} M_{a}(t)
&\geq& 8\delta'\int_{|x|<R}P(u)\, dx
+ \mathcal{O}\!\left(\frac{1}{R^{2}}+\frac{1}{R^{b}}+\mathcal{O}_R\right).
\end{eqnarray}

Finally,
\begin{eqnarray}
8\delta'\int_{0}^{T}\int_{|x|<R}P(u)\, dx\, dt
&\lesssim& M_{a}(t) + \mathcal{O}\!\left(\frac{1}{R^{2}}+\frac{1}{R^{b}}+\mathcal{O}_R\right)T \\
&\lesssim& R + \mathcal{O}\!\left(\frac{1}{R^{2}}+\frac{1}{R^{b}}+\mathcal{O}_R\right)T.
\end{eqnarray}

By the definition of $\chi_{R}$, we conclude (\ref{pro5.3a}).
\begin{eqnarray*}
\frac{1}{T}\int_{0}^{T}\int_{|x|<R}P(u)dx&\lesssim&\frac{R}{T}+\mathcal{O}\left(\frac{1}{R^{2}}+\frac{1}{R^{b}}+\mathcal{O}_R\right).
\end{eqnarray*}
Let $T_{n}\to+\infty$ and apply (\ref{pro5.3a}) with $R_{n}=T_{n}^{1/3}$. We obtain
\begin{eqnarray*}
\frac{1}{T_{n}}\int_{0}^{T_{n}}\int_{|x|<R_{n}}P(u)\,dx\,dt 
&\lesssim& T_{n}^{-2/3}\;\longrightarrow\;0.
\end{eqnarray*}
By the fundamental theorem of calculus, it follows that there exist sequences $t_n\to+\infty$ and $R_n\to\infty$ such that
\begin{equation*}
\int_{|x|<R_{n}}P(u(t_n))\,dx \;\longrightarrow\; 0.
\end{equation*}
\end{proof}

\ Finally, we conclude this section by proving the main theorem.
\begin{proof}[\bf Proof of Theorem (\ref{mainresulttheorem})] Take a sequence of times $t_n \to +\infty$, and a sequence $R_n \to 0$ as in Proposition \ref{prop5.3}. Fix $\varepsilon>0$ and $R \ge 0$ as in Proposition \ref{Scattering}. Choose $n$ sufficiently large so that $R_n \ge R$. Then, by Hölder’s inequality, we have
\begin{equation}
\int_{|x|<R} |u(t_n, x)|^2 \, dx \lesssim R^{\frac{2b + 3p - 6}{p}} 
\left( \int_{|x|<R} |x|^{-b} |u(t_n, x)|^p \, dx \right)^{2/p}.
\end{equation}

Next, recall that
\begin{align*}
\left( \int_{|x|<R} |x|^{-b} |u(t_n, x)|^p \, dx \right)^2 
&\lesssim (2R)^{3-\gamma} \int_{|x|<R} \int_{|y|<R} 
\frac{|x|^{-b} |u(t_n, x)|^p \, |y|^{-b} |u(t_n, y)|^p}{|x-y|^{3-\gamma}} \, dx \, dy \\
&\lesssim (2R)^{3-\gamma} \int_{|x|<R} |x|^{-b} |u(t_n, x)|^p 
\left( \int_{|y|<R} \frac{|y|^{-b} |u(t_n, y)|^p}{|x-y|^{3-\gamma}} \, dy \right) dx \\
&\lesssim (2R)^{3-\gamma} \int_{|x|<R} P(u(t_n, x)) \, dx \to 0 \quad \text{as } n \to \infty.
\end{align*}

This implies that
\begin{equation}
\int_{|x|<R} |u(t_n, x)|^2 \, dx \to 0 \quad \text{as } n \to \infty.
\end{equation}

Therefore, Proposition \ref{Scattering} implies that $u$ scatters in $H^1(\mathbb{R}^3)$ forward in time.

\end{proof}

\section*{Declarations}

\noindent\textbf{Ethical Approval:} Not applicable. \\
\noindent\textbf{Funding:} C.M.G. was partially supported by Conselho Nacional de Desenvolvimento Científico e Tecnológico (CNPq) and Fundação de Amparo à Pesquisa do Estado do Rio de Janeiro (FAPERJ), Brazil. S.S. was financed in part by the Coordenação de Aperfeiçoamento de Pessoal de Nível Superior – Brazil (CAPES) – Finance Code 001. G.P. was partially supported by Conselho Nacional de Desenvolvimento Científico e Tecnológico (CNPq), Brazil. \\
\noindent\textbf{Authors' Contributions:} The authors contributed equally to the preparation of this paper. \\
\noindent\textbf{Availability of Data and Materials:} No new data or materials were used in preparing this paper.

\bibliographystyle{abbrv}
\bibliography{bib}	

\end{document}